\documentclass[12pt,leqno,fleqn]{amsart}
\usepackage{amssymb}

\newtheorem{theorem}{Theorem}[section]
\newtheorem{lemma}[theorem]{Lemma}
\newtheorem{remark}[theorem]{Remark}
\newtheorem{corollary}[theorem]{Corollary}

\newtheorem{definition}[theorem]{Definition}
\newtheorem{proposition}[theorem]{Proposition}

\numberwithin{equation}{section}

\newcommand{\pp}{{\mathbb P}}

\newcommand{\rr}{{\mathbb R}}

\newcommand{\nn}{{\mathbb N}}

\begin{document}
\author[A. Aliouche]{Abdelkrim Aliouche}
\address{A. Aliouche, Department of Mathematics\\
University of Larbi Ben M'Hidi\\
Oum-El Bouaghi, 04000, Algeria}
\email{alioumath@yahoo.fr}
\author[C. Simpson]{Carlos Simpson}
\address{C. Simpson, CNRS, Laboratoire J. A. Dieudonn\'{e}, UMR 6621 \\
Universit\'e de Nice-Sophia Antipolis\\
06108 Nice, Cedex 2, France}
\email{carlos@unice.fr}
\title[Fixed points and lines]{Fixed
points and lines in $2$-metric spaces}

\begin{abstract}
We consider bounded $2$-metric spaces satisfying an additional axiom, and show that
a contractive mapping has either a fixed point or a fixed line. 
\end{abstract}

\maketitle

\section{Introduction}
\label{intro}

G\"ahler introduced in the 1960's the notion of $2$-metric space \cite{Gahler1} \cite{Gahler2} \cite{Gahler3},
and several authors have studied the question of fixed point theorems for mappings on such spaces. 
A $2$-metric is a function $d(x,y,z)$ symmetric under permutations, satisfying the {\em
tetrahedral inequality}
\[
d(x,y,z)\leq d(x,y,a)+d(x,a,z)+d(a,y,z)\mbox{  for all  }x,y,z,a\in X.
\]
as well as conditions (Z) and (N) which will be recalled below. In the prototypical example, $d(x,y,z)$ is
the area of the triangle spanned by $x,y,z$. 

This notion has been considered by several authors (see \cite{FreeseCho}), 
who have notably generalized Banach's principle to obtain
fixed point theorems, for example White \cite{White}, Iseki \cite{Iseki}, Rhoades \cite{Rhoades},
Khan \cite{Khan}, Singh, Tiwari and Gupta \cite{SinghTiwariGupta}, Naidu and Prasad \cite{NaiduPrasad}, Naidu \cite{Naidu} and Zhang \cite{LiuZhang}, 
Abd El-Monsef, Abu-Donia, Abd-Rabou \cite{AbdElMonsefEtAl}, Ahmed \cite{Ahmed} and others. 
The contractivity conditions used in these works are usually of the form
\[
d(F(x),F(y),a)\leq \ldots 
\]
for any $a\in X$. We may think of this as meaning that $d(x,y,a)$ is a family of distance-like functions of $x$ and $y$, indexed by $a\in X$.
This interpretation intervenes in our transitivity condition (Trans) below. However, Hsiao has
shown that these kinds of contractivity conditions don't have a wide range of
applications, since they imply colinearity of the sequence of iterates starting
with any point \cite{Hsiao}. We thank B. Rhoades for pointing this out to us. 

There have also been several different notions of a space together with a function of $3$-variables. For example, 
Dhage \cite{Dhage} introduced the
concept of $D$-metric space and proved the existence of a unique fixed point
of a self-mapping satisfying a contractive condition. Dhage's definition uses the 
symmetry and tetrahedral axioms present in G\"ahler's definition, but 
includes the {\em coincidence} axiom that $d(x,y,z)=0$ if and only if $x=y=z$. 

A sequence $\{x_{n}\}$ in a $D$-metric space $(X,d)$ is said by Dhage to be
convergent to an element $x\in X$ \ (or $d$-convergent) \cite{Dhage} if
given $\epsilon >0$, there exists an $N\in \nn$
such that $d(x_{m},x_{n},x)<\epsilon $ for all $m,n\geq N$.
He calls a sequence $\{x_{n}\}$ in a $D$-metric space $(X,d)$ Cauchy
(or $d$-Cauchy) \cite{Dhage} if given $\epsilon >0$, there exists an $N\in \nn$
such that $d(x_{n},x_{m},x_{p})<\epsilon $ for all $n,m,p\geq N$.

These definitions, distinct from those used by G\"ahler et al,
motivate the definition of the property 
$LIM(y,(x_i))$ in Definition \ref{def-lim} and studied
in Theorem \ref{possibilities} below.

The question of fixed-point theorems on such spaces has proven to be somewhat delicate
\cite{MustafaSims1}. 
Mustafa and Sims introduced a notion of $G$-metric space \cite{MustafaSims} \cite{MustafaSims2}, 
in which the tetrahedral inequality is replaced by an inequality involving repetition of indices. 
In their point of view the function $d(x,y,z)$ is thought of as representing the perimeter of a triangle. 

The question of fixed points for mappings on $G$-metric spaces has been considered by
Abbas-Rhoades \cite{AbbasRhoades}, Mustafa and co-authors \cite{MustafaObiedatAwawdeh}, 
\cite{MustafaShatawaniBataineh}. This is not an exhaustive description of the large literature on this subject. 

In the present paper, we return to the notion of $2$-metric space.
The basic philosophy is that since a $2$-metric measures area, a contraction, that is a map $F$ such that
\begin{equation}
\label{tc}
d(F(x),F(y),F(z)) \leq kd(x,y,z)
\end{equation}
for some $k<1$, should
send the space towards a configuration of zero area, which is to say a line. 
Results in this direction, mainly 
in the
case of the euclidean triangle area, have been obtained 
by a small circle of authors
starting with Zamfirescu \cite{Zamfirescu} and
Daykin and Dugdale \cite{DaykinDugdale}, who called such mappings
``triangle-contractive''. Notice that 
the triangle-contractive condition is different from the ones discussed
previously, in that $F$ is applied to all three variables.

Zamfirescu, Daykin and Dugdale obtained results saying that the set
of limit points of iterates of such maps are linear, giving
under some hypotheses either fixed points or fixed lines. 
Subsequent papers in this direction include Rhoades \cite{RhoadesFixtures},
Ang-Hoa \cite{AngHoa1} \cite{AngHoa2}, 
Dezs\"{o}-Mure\c{s}an \cite{DezsoMuresan},  
and Kapoor-Mathur \cite{KapoorMathur}. The paper of 
Dezs\"{o} and Mure\c{s}an envisions the extension of the theory to the
case of $2$-metric spaces, but most of their results concern the euclidean case or the case of a $2$-normed linear space. 

In order to obtain a treatment which applies to more general
$2$-metric spaces, yet 
always assuming that $d$ is globally bounded (B), we add an additional quadratic axiom (Trans) to the original definition of $2$-metric. 
Roughly speaking this axiom says that if $x,y,z$ are approximately
colinear, if $y,z,w$ are approximately colinear, and if $y$ and $z$
are far enough apart, then $x,z,w$ and $x,y,w$ are approximately colinear. 
The axiom (Trans) will be shown to hold
in the example $X= S^2$ where $d(x,y,z)$ is given by a determinant (Section \ref{example}), which 
has appeared in \cite{MiczkoPalczewski}, as well as for the standard area $2$-metric
on ${\mathbb{R}}^n$.
The abbreviation comes from the
fact that (Trans) implies transitivity of the relation of colinearity,
see Lemma \ref{coltrans}.  
This axiom allows us to consider a notion of {\em fixed line} of a mapping $F$ which 
is contractive in the sense of \eqref{tc}.  
With these hypotheses on $d$ and
under appropriate compactness assumptions we prove that such a mapping has either a fixed point or a fixed line.

In the contractivity condition \eqref{tc}, the function $F$
is applied to all three variables. Consequently, it turns out that 
there exist many mappings satisfying our contractivity condition,
but not the triviality observed by Hsiao \cite{Hsiao}. 
Some examples will be discussed in
Section \ref{examplemaps}. 
In the example of $S^2$ with the norm of determinant $2$-metric, one can take
a neighborhood of the equator which contracts towards the equator, composed with a 
rotation. This will have the equator as fixed line, but no fixed point, and
the successive iterates of a given point will not generally be colinear. 
Interesting examples of $2$-metrics on manifolds are obtained from embedding in 
${\mathbb{R}}^n$ and pulling back the standard area $2$-metric. The properties
in a local coordinate chart depend in some way on the curvature of the embedded
submanifold. We consider a first case of patches on $S^2$ in the euclidean
${\mathbb{R}}^3$. These satisfy an estimate (Lemma \ref{s2bounds}) which
allows to exhibit a ``flabby'' family of contractible mappings depending
on functional parameters (Proposition \ref{s2maps}). This shows that 
in a strong sense the objection of \cite{Hsiao} doesn't apply.

The first section of the paper considers usual metric-like functions of two
variables, pointing out that the classical triangle inequality may be weakened
in various ways. A bounded $2$-metric leads naturally to such a distance-like function
$\varphi (x,y)$ but we also take the opportunity to sketch some directions
for fixed point results in this general context, undoubtedly in the same direction
as \cite{Peppo} but more elementary.

As a small motivation to readers more oriented towards abstract category theory, 
we would like to point out that a metric space (in the classical sense) may be considered
as an enriched category: the ordered set $({\mathbf{R}}_{\geq 0}, \leq )$
considered as a category has a monoidal structure $+$, and a metric space is just an
$({\mathbf{R}}_{\geq 0}, \leq ,+)$-enriched category. We have learned this observation from
Leinster and Willerton
\cite{LeinsterWillerton} although it was certainly known before. An interesting
question is, what categorical structure corresponds to the notion of $2$-metric?

\section{Asymmetric triangle inequality}
\label{asymmtriangle}

Suppose $X$ is a set together with a function $\varphi (x,y)$ 
defined for $x,y\in X$ such that: 
\newline
(R)---$\varphi (x,x)= 0$; \newline
(S)---$\varphi (x,y)= \varphi (y,x)$; \newline
(AT)---for a constant $C\geq 1$ saying
\begin{equation*}
\varphi (x,y)\leq \varphi (x,z)+ C\varphi (z,y).
\end{equation*}
In this case we say that $(X,\varphi )$ satisfies the {\em asymmetric triangle inequality}.

It follows that $\varphi (x,y)\geq 0$ for all $x,y\in X$. Furthermore, if we
introduce a relation $x\sim y$ when $\varphi (x,y)=0$, then the three axioms
imply that this is an equivalence relation, and furthermore when $x\sim
x^{\prime}$ and $y\sim y^{\prime}$ then $\varphi (x,y)=\varphi
(x^{\prime},y^{\prime})$. Thus, $\varphi$ descends to a function on the
quotient $X/\sim $ and on the quotient it has the property that 
$\varphi (\overline{x},\overline{y})=0\Leftrightarrow \overline{x}=\overline{y}$. In
view of this discussion it is sometimes reasonable to add the {\em strict reflexivity}
axiom \newline
(SR)---if $\varphi (x,y)=0$ then $x=y$.

B. Rhoades pointed out to us that
the asymmetric triangle inequality implies the property $\varphi (x,y)\leq \gamma 
(\varphi (x,z)+\varphi (z,y))$ of a {\em quasidistance} used by Peppo \cite{Peppo} 
and it seems
likely that the following discussion could be a consequence of her fixed point
result for $(\varphi , i,j,k)$-mappings, although that deduction 
doesn't seem immediate.

It is easy to see for $(X,\varphi )$ satisfying the asymmetric triangle inequality,
that the notion of limit for the distance function $\varphi $ makes sense,
similarly the notion of Cauchy sequence for $\varphi $ makes sense, and we
can say that $(X,\varphi )$ is \emph{complete} if every Cauchy sequence has
a limit.  If a
sequence has a limit then it is Cauchy. The function $\varphi $ is
continuous, i.e., it transforms limits into limits. 
If furthermore the strictness axiom (SR) satisfied, then 
limit is unique.

A {\em point of accumulation}
of a sequence $(x_i)_{i\in \nn}$ is a limit of a subsequence, 
that is to say a point $y$ such that there exists a subsequence 
$(x_{i(j)})_{j\in \nn}$ with $i(j)$ increasing, such that $y=\lim _{j\rightarrow \infty}x_{i(j)}$. A set $X$ provided
with a distance function satisfying the asymmetric triangle inequality (i.e. (R), (S) and (AT)), is {\em compact}
if every sequence has a point of accumulation. In other words, every sequence
admits a convergent subsequence.  This notion should perhaps be called ``sequentially compact'' 
but it is the only compactness notion which will be used in what follows. 

\begin{lemma}
\label{varphicontractingmap} Suppose $(X,\varphi )$ satisfies
the asymmetric triangle inequality with the constant $C$. Suppose 
$F:X\rightarrow X$ is a map such that 
$\varphi (Fx,Fy)\leq k\varphi (x,y)$ with $k<(1/C)$. Then for any $x\in X$
the sequence $\{F^{i}(x)\}$ is Cauchy. If $(X,\varphi )$ is complete and strictly reflexive 
then its limit is the unique fixed point of $F$.
\end{lemma}

\begin{proof}
Let $x_{0}$ be an arbitrary point in $X$ and $\{x_{n}\}$ the sequence
defined by $x_{n+1}=F(x_{n})=F^{n}(x_{0})$ for all positive integer $n$. We
have 
\begin{equation*}
\varphi (x_{n+1},x_{n})=\varphi (Fx_{n},Fx_{n-1})\leq k\varphi
(x_{n},x_{n-1}).
\end{equation*}
By induction, we obtain 
\begin{equation*}
\varphi (x_{n+1},x_{n})\leq k^{n}\varphi (x_{0},x_{1})
\end{equation*}
Using the \emph{asymmetric triangle inequality} several times we get for all
positive integers $n,m$ such that $m>n$ 
\begin{equation*}
\varphi (x_{n},x_{m})\leq \varphi (x_{n},x_{n+1})+C\varphi
(x_{n+1},x_{n+2})+C^{2}\varphi (x_{n+2},x_{n+3})+...
\end{equation*}
\begin{equation*}
...+C^{m-n-1}\varphi (x_{m-1},x_{m}).
\end{equation*}
Then 
\begin{equation*}
\varphi (x_{n},x_{m})\leq \ k^{n}\varphi (x_{0},x_{1})+Ck^{n+1}\varphi
(x_{0},x_{1})+C^{2}k^{n+2}\varphi (x_{0},x_{1})+
\end{equation*}
\begin{equation*}
...+C^{m-n-1}k^{m-1}\varphi (x_{0},x_{1}).
\end{equation*}
Therefore 
\begin{equation*}
\varphi (x_{n},x_{m})\leq \
(1+Ck+C^{2}k^{2}+....+C^{m-n-1}k^{m-n-1})k^{n}\varphi (x_{0},x_{1})
\end{equation*}
and so 
\begin{equation*}
\varphi (x_{n},x_{m})<\frac{k^{n}}{1-Ck}\varphi (x_{0},x_{1})
\end{equation*}
Hence, the sequence $\{x_{n}\}$ is Cauchy. Since $(X,\varphi )$ is complete,
it converges to some $x\in X$. Now, we show that $z$ is a fixed point of $F$. 
Suppose not. Then 
\begin{equation*}
\varphi (Fz,Fx_{n})\leq k\varphi (z,x_{n-1})
\end{equation*}
As $n$ tends to infinity we get $z=Fz$ using (SR). The uniqueness of $z$ follows easily.
\end{proof}

\begin{corollary}
\label{vcmcor} Suppose $(X,\varphi )$ satisfies the asymmetric triangle
inequality, is strictly reflexive and  complete. If $F:X\rightarrow X$ is a map such
that $\varphi (Fx,Fy)\leq k\varphi (x,y)$ with $k<1$, then $F$ has a unique
fixed point.
\end{corollary}

\begin{proof}
Since $k<1$ there exists $a_0\geq 1$ such that $k^a<(1/C)$ for any $a\geq
a_0 $. Then the previous lemma applies to $F^a$ whenever $a\geq a=0$, and 
$F^a$ has a unique fixed point $z_a$. Choose $b\geq a_0$ and let $z_b$ be the
unique fixed point of $F^b$. Then 
\begin{equation*}
F^{ab}(z_b) = (F^b)^a(z_b)=z_b,
\end{equation*}
but also 
\begin{equation*}
F^{ab}(z_a) = (F^a)^b(z_a)=z_a.
\end{equation*}
Thus $z_a$ and $z_b$ are both fixed points of $F^{ab}$; as $ab\geq a_0$ its
fixed point is unique so $z_a=z_b$. Apply this with $b=a+1$, so 
\begin{equation*}
F(z_a)= F(F^a (z_a)) = F^b(z_a)=F^b(z_b)=z_b = z_a.
\end{equation*}
Thus $z_a$ is a fixed point of $F$. If $z$ is another fixed point of $F$
then it is also a fixed point of $F^a$ so $z=z_a$; this proves uniqueness.
\end{proof}

\subsection{Triangle inequality with cost}
\label{tricost}

If $d(x,y,z)$ is a function of three variables, the ``triangle inequality with cost'' is
\begin{equation}
\label{tic}
\varphi (x,y)\leq \varphi (x,z)+\varphi (z,y)+d(x,y,z).
\end{equation}
This enters into Lemma \ref{trianglevariants} below. 

We mention in passing 
a ``triangle inequality with multiplicative
cost'': suppose given a function $\varphi (x,y)$ plus a function of $3$ variables 
$\psi (x,y,z)$ such that 
\begin{equation}
\label{timc}
\varphi (x,y)\leq (\varphi (x,z)+\varphi (z,y))e^{\psi (x,y,z)}.
\end{equation}
Assume also that $\varphi $ is invariant under transposition, with $\varphi
(x,y)=0\Leftrightarrow x=y$ and that $\psi $ is bounded above and below. We
can define limits and Cauchy sequences, hence completeness and the function 
$\varphi $ is continuous.
The following fixed point statement is not used elsewhere but seems interesting
on its own. 

\begin{proposition}
\label{fixedpointmultcost} Suppose given $\varphi ,\psi $ satisfying the
triangle inequality with multiplicative cost \eqref{timc}
as above. If $F$
is a map such that 
\[ \varphi (F(x),F(y))\leq k\varphi (x,y) \mbox{   and   } 
\]
\[
\psi
(F(x),F(y),F(z))\leq k\psi (x,y,z)
\] 
whenever both sides are positive, then
we get a Cauchy sequence $F^{k}(x)$. If $(X,\varphi )$ is complete then the
limit of this Cauchy sequence is the unique fixed point of $F$.
\end{proposition}

\begin{proof}
Let $x_{0}$ be an arbitrary point in $X$ and $\{x_{n}\}$ the sequence
defined by $x_{n+1}=F(x_{n})=F^{n}(x_{0})$ for all positive integer $n$. We
have
\[
\varphi (x_{n},x_{m}) 
\leq (\varphi (x_{n},x_{n+1})+\varphi
(x_{n+1},x_{m}))e^{\psi (x_{n},x_{m},x_{n+1})} 
\]
\[
\leq k^{n}\varphi (x_{0},x_{1})e^{\psi (x_{n},x_{m},x_{n+1})}+\varphi
(x_{n+1},x_{m})e^{\psi (x_{n},x_{m},x_{n+1})} 
\]
\[
\leq k^{n}\varphi (x_{0},x_{1})e^{\psi (x_{n},x_{m},x_{n+1})}+ 
\]
\[
(\varphi (x_{n+1},x_{n+2})+\varphi (x_{n+2},x_{m}))e^{\psi
(x_{n+1},x_{m},x_{n+2})+{\psi (x_{n},x_{m},x_{n+1})}} 
\]
\[
\leq k^{n}\varphi (x_{0},x_{1})e^{\psi (x_{n},x_{m},x_{n+1})}+ 
\]
\[
k^{n+1}\varphi (x_{0},x_{1})e^{\psi (x_{n+1},x_{m},x_{n+2})+\psi
(x_{n},x_{m},x_{n+1})}+ \]
\[
\varphi (x_{n+2},x_{m})e^{\psi (x_{n+1},x_{m},x_{n+2})+\psi
(x_{n},x_{m},x_{n+1})} 
\]
\[
\leq k^{n}\varphi (x_{0},x_{1})e^{\psi (x_{n},x_{m},x_{n+1})}+ 
\]
\[
k^{n+1}\varphi (x_{0},x_{1})e^{\psi (x_{n+1},x_{m},x_{n+2})+\psi
(x_{n},x_{m},x_{n+1})}+ 
\]
\[
k^{n+2}\varphi (x_{0},x_{1})e^{\psi (x_{n+2},x_{m},x_{n+3})+\psi
(x_{n+1},x_{m},x_{n+2})+\psi (x_{n},x_{m},x_{n+1})}+ 
\]
\[
\varphi (x_{n+3},x_{m})e^{\psi (x_{n+2},x_{m},x_{n+3})+\psi
(x_{n+1},x_{m},x_{n+2})+\psi (x_{n},x_{m},x_{n+1})} 
\]
\[
\leq k^{n}\varphi (x_{0},x_{1})e^{\psi (x_{n},x_{m},x_{n+1})}+ 
\]
\[
k^{n+1}\varphi (x_{0},x_{1})e^{\psi (x_{n+1},x_{m},x_{n+2})+\psi
(x_{n},x_{m},x_{n+1})}+ 
\]
\[
k^{n+2}\varphi (x_{0},x_{1})e^{\psi (x_{n+2},x_{m},x_{n+3})+\psi
(x_{n+1},x_{m},x_{n+2})+\psi (x_{n},x_{m},x_{n+1})}+...+ 
\]
\[
k^{m-1}\varphi (x_{0},x_{1})e^{\psi (x_{n},x_{m},x_{n+1})+\psi
(x_{n+1},x_{m},x_{n+2})+...+\psi
(x_{m-2},x_{m},x_{m-1})} 
\]
\[
\leq k^{n}\varphi (x_{0},x_{1})(e^{k^{n}\psi
(x_{0},x_{p},x_{1})}+ke^{k^{n}\psi (x_{0},x_{p},x_{1})+k^{n+1}\psi
(x_{0},x_{p},x_{1})}+ 
\]
\[
k^{2}e^{k^{n}\psi (x_{0},x_{p},x_{1})+k^{n+1}\psi
(x_{0},x_{p},x_{1})+k^{n+2}\psi (x_{0},x_{p},x_{1})}+ 
\]
\[
\ldots + 
\]
\[
k^{m-n-1}e^{k^{n}\psi (x_{0},x_{p},x_{1})+k^{n+1}\psi
(x_{0},x_{p},x_{1})+...+k^{m-1}\psi
(x_{0},x_{p},x_{1})}) 
\]
\[
\leq k^{n}\varphi
(x_{0},x_{1})(e^{Mk^{n}}+ke^{M(k^{n}+k^{n+1})}+k^{2}e^{M(k^{n}+k^{n+1}+k^{n+2})}+
\]
\[
\ldots + 
\]
\[
k^{m-n-1}e^{M(k^{n}+k^{n+1}+k^{n+2}+...+k^{m-1})})
\]
since $\psi $ is bounded. Hence, the sequence $\{x_{n}\}$ is Cauchy. Since 
$(X,\varphi )$ is complete, it converges to some $x\in X$. The rest of the
proof follows as in Lemma \ref{varphicontractingmap}.
\end{proof}

\section{Bounded $2$-metric spaces}
\label{sec-metr2}

G\"ahler defined the notion of $2$-metric space to be a set $X$ with function 
$d:X^3\rightarrow {\mathbb{R}}$ denoted $(x,y,z)\mapsto d(x,y,z)$
satisfying the following axioms \cite{Gahler1} \cite{Gahler2} \cite{Gahler3}:

\noindent (Sym)---that $d(x,y,z)$ is invariant under permutations of the
variables $x,y,z$.

\noindent (Tetr)---for all $a,b,c,x$ we have 
\begin{equation*}
d(a,b,c)\leq d(a,b,x)+d(b,c,x)+d(a,c,x).
\end{equation*}

\noindent (Z)---for all $a,b$ we have $d(a,b,b)=0$.

\noindent (N)---for all $a,b$ there exists $c$ such that $d(a,b,c)\neq 0$. 

One can think of $d(x,y,z)$ as measuring how far
are $x,y,z$ from being ``aligned'' or ``colinear''.

The $2$-metric spaces $(X,d)$ have been the subject of much study, 
see \cite{AbbasRhoades} and \cite{AbdElMonsefEtAl} for example.
The prototypical example of a $2$-metric space is obtained
by setting $d(x,y,z)$ equal to the area of the triangle spanned by $x,y,z$.

Assume that the $2$-metric is {\em bounded}, and by rescaling the bound can be supposed equal to $1$:

\noindent (B)---the function is bounded by $d(x,y,z)\leq 1$ for all 
$x,y,z\in X$.

Define the \emph{associated distance} by 
\begin{equation*}
\varphi (x,y):= \sup _{z\in X} d(x,y,z).
\end{equation*}

\begin{lemma}
\label{positive}
We have $d(x,y,z)\geq 0$ and hence $\varphi (x,y)\geq 0$. Also $\varphi
(x,x)=0$ and $\varphi (x,y)=\varphi (y,x)$.
\end{lemma}

\begin{proof}
Applying the axiom (Tetr) with $b=c$, we get 
\begin{equation*}
d(a,b,b)\leq d(a,b,x) + d(b,a,x)+d(a,b,x).
\end{equation*}
By the axiom (Z) and the symmetry of $d$ we obtain $d(a,b,x)\geq 0$ and so 
$d(x,y,z)\geq 0$. Then, $\varphi (x,y)\geq 0$. Symmetry of $\varphi$ follows
from invariance of $d$ under permutations (Sym).
\end{proof}

\begin{lemma}
\label{trianglevariants}
We have the triangle inequality with cost \eqref{tic} 
\begin{equation*}
\varphi (x,y) \leq \varphi (x,z)+\varphi (z,y) +d(x,y,z).
\end{equation*}
Therefore 
\begin{equation*}
\varphi (x,y) \leq \varphi (x,z)+\varphi (z,y) + \min (\varphi (x,z),
\varphi (z,y))
\end{equation*}
and hence the asymmetric triangle inequality (AT)
\begin{equation*}
\varphi (x,y)\leq \varphi (x,z)+2\varphi (z,y).
\end{equation*}
\end{lemma}

\begin{proof}
We have 
\begin{equation*}
d(x,y,z_{0})\leq d(x,y,z) + d(y,z_{0},z)+d(x,z_{0},z)
\end{equation*}
\begin{equation*}
\leq \varphi (x,z)+\varphi (z,y) +d(x,y,z).
\end{equation*}
For the next statement, note that by definition 
\begin{equation*}
d(x,y,z)\leq \min (\varphi (x,z), \varphi (z,y)),
\end{equation*}
and for the last statement, $\min (\varphi (x,z), \varphi (z,y))\leq \varphi
(z,y)$.
\end{proof}

In particular the distance $\varphi$ satisfies the axioms (R), (S) and (AT)
of Section \ref{asymmtriangle}. This allows us to speak of limits, Cauchy sequences,
points of accumulation, completeness and compactness, see also \cite{Peppo}.  
For clarity it will usually
be specified that these notions concern the function $\varphi$. 
Axiom (N) for $d$ is equivalent to
strict reflexivity (SR) for $\varphi$; if this is not assumed from the start, it 
can be fixed as follows. 

\subsection{Nondegeneracy}

It is possible to start without supposing the nondegeneracy axiom (N),
define an equivalence relation, and obtain a $2$-metric on the quotient satisfying (N).
For the next lemma and its corollary,
we assume that $d$ satisfies all of
(Sym), (Tetr), (Z), (B), but not necessarily (N).

\begin{lemma}
\label{dphi}
If $a,b,x,y$ are any points then
\[
|d(a,b,x)-d(a,b,y)|\leq 2 \varphi (x,y).
\]
\end{lemma}

\begin{proof}
By condition (Tetr),
\[
d(a,b,y)\leq d(a,b,x)+d(b,y,x)+d(a,y,x) \leq d(a,b,x)+2\varphi (x,y).
\]
The same in the other direction gives the required estimate. 
\end{proof}

\begin{corollary}
If $x,y$ are two points with $\varphi (x,y)=0$ then for any $a,b$
we have $d(a,b,x)=d(a,b,y)$. Therefore, if $\sim$ is the equivalence relation
considered in the second paragraph of Section \ref{asymmtriangle},
the function $d$ descends to a function $(X/\sim )^3\rightarrow \rr$
satisfying the same properties but in addition its associated distance function
is strictly reflexive and $d$ satisfies (N). 
\end{corollary}
\begin{proof}
For the first statement, apply the previous lemma.  This invariance
applies in each of the three arguments since $d$ is invariant under permutations, which in
turn yields the descent of $d$ to a function on $(X/\sim )^3$. The associated distance
function is the descent of $\varphi$ which is strictly reflexive. 
\end{proof}

In view of this lemma, we shall henceforth assume that $\varphi$ satisfies (SR)
or equivalently $d$ satisfies (N) too.
In particular the limit of a sequence is unique if it exists. 
 
\subsection{Surjective mappings}

If $F$ is surjective, then a boundedness condition for $d$ implies the
same for $\varphi$. Since we are assuming that $d$ is globally bounded (condition (B)),
a surjective mapping cannot be strictly contractive:

\begin{lemma}
Suppose $F:X\rightarrow X$ is a map such that 
\begin{equation*}
d(F(x),F(y),F(z))\leq kd(x,y,z)
\end{equation*}
for some constant $k>0$. If $F$ is surjective then $\varphi (F(x),F(y))\leq
k\varphi (x,y)$. The global boundedness condition implies that $k\geq 1$ in 
this case. 
\end{lemma}

\begin{proof}
Suppose $x,y\in X$. For any $z\in X$, choose a preimage $w\in X$ such that 
$F(w)=z$ by surjectivity of $F$. Then 
$$
d(F(x),F(y),z) = d(F(x),F(y),F(w))\leq kd(x,y,w) \leq k\varphi (x,y).
$$
It follows that 
\begin{equation*}
\varphi (F(x),F(y))=\sup _{z\in X} d(F(x),F(y),z) \leq k\varphi (x,y).
\end{equation*}
Suppose now that $k<1$. Let $B$ be the supremum of $\varphi (x,y)$ for $x,y\in X$.
Then $0<B<1$ by conditions (N) and (B).  Therefore there exist $x,y$ such that
$kB < \varphi (x,y)$, but this contradicts the existence of $u$ and $v$ such that
$F(u)=x$ and $F(v)=y$. This shows that $k\geq 1$. 
\end{proof}

\section{Colinearity}
\label{colinfunc}

Consider a bounded $2$-metric space $(X,d)$ , 
that is to say satisfying
axioms (Sym), (Tetr), (Z), (N) and (B),
and require the following additional {\em transitivity axiom}:

\noindent (Trans)---for all $a,b,c,x,y$ we have 
\begin{equation*}
d(a,b,x) d(c,x,y)\leq d(a,x,y)+d(b,x,y).
\end{equation*}

In Section \ref{example} below we will see that the standard area function,
as well as a form of geodesic
area function on ${\mathbb{RP}^2}$, satisfy this additional axiom. 
The terminology ``transitivity'' comes from the fact that this
condition implies a transitivity property of the
relation of colinearity, see Lemma \ref{coltrans} below.

The term $d(c,x,y)$ may be replaced by its ${\rm sup}$ over $c$ which is $\varphi (x,y)$. 
If we think of $d(a,b,x)$ as being a family of distance-like functions of $a$ and $b$,
indexed by $x\in X$, (Trans) can be rewritten 
\[
d(a,b,x)\leq (d(a,y,x) + d(y,b,x))\varphi (x,y)^{-1}
\]
for $y\neq x$. This formulation
may be related to the notion of ``triangle inequality with multiplicative cost'' \eqref{timc} discussed in Section \ref{tricost}.

\begin{definition}
\label{def-colinear}
Say that $(x,y,z)$ are \emph{colinear} if $d(x,y,z)=0$.
A \emph{line} is a maximal subset $Y\subset X$ consisting of colinear points,
that is to say satisfying 
\begin{equation}
\label{subline}
\forall x,y,z\in Y, \;\;
d(x,y,z)=0.
\end{equation}
\end{definition}

The colinearity condition is symmetric under permutations by (Sym). 

\begin{lemma}
\label{coltrans}
Using all of the above axioms including (N) and assumption (Trans),
colinearity satisfies the following transitivity property:
if $x,y,z$ are colinear, $y,z,w$ are colinear, and $y\neq z$, then 
$x,y,w$ and $x,z,w$ are colinear.
\end{lemma}
\begin{proof}
By (N), $\varphi (y,z)\neq 0$, then use the above version of (Trans) rewritten after
some permutations as
\[
d(x,y,w)\leq (d(x,y,z) + d(y,z,w))\varphi (y,z)^{-1}.
\]
This shows that $x,y,w$ are colinear. Symmetrically, the same for $x,z,w$.
\end{proof}

A line is nonempty, by maximality since $d(y,y,y)=0$ by (Z).

\begin{lemma}
\label{lineunique}
If $x\neq y$ are two
points then there is a unique line $Y$ containing $x$ and $y$, and $Y$
is the set of points $a$ colinear with $x$ and $y$, i.e. such that $d(a,x,y)=0$.
\end{lemma}

\begin{proof}
The set $\{ x,y\}$ satisfies Condition \eqref{subline}, so there is at least one
maximal such set $Y$ containing $x$ and $y$. Choose one such $Y$.   
If $a\in Y$ then automatically $d(a,x,y)=0$. 

Suppose $d(a,x,y)=0$. By Lemma \ref{coltrans}, $d(a,x,u)=0$
for any $u\in Y$.

Now suppose $u,v\in Y$. If $u=x$ then the preceding shows that
$d(a,u,v)=0$. If $u\neq x$ then, since $d(x,u,v)=0$ and $d(a,x,u)=0$,
Lemma \ref{coltrans} implies that $d(a,u,v)=0$. This shows that $a$ is
colinear with any two points of $Y$. In particular, $Y\cup \{ a\}$ also satisfies
Condition \eqref{subline} so by maximality, $a\in Y$. This shows that $Y$
is the set of points $a$ such that $d(a,x,y)=0$, which characterizes it uniquely. 
\end{proof}

The notions of colinearity and lines come from the geometric examples of
$2$-metrics which will be discussed in Section \ref{example} below. 
It should be pointed out that there can be interesting examples of $2$-metric
spaces which don't satisfy the transitivity condition of Lemma \ref{coltrans} and
which therefore don't satisfy Axiom (Trans). The remainder of our discussion doesn't apply 
to such examples. 

We assume Axiom (Trans) from now on. It
allows us to look at the question of fixed subsets of
a contractive mapping $F$ when $F$ is not surjective.
In addition to the possibility of having a fixed point, there will also be the
possibility of having a fixed line. We see in examples below that this can happen.

\begin{definition}
\label{def-lim}
Consider a sequence of points $x_i\in X$. The property $LIM(y,(x_i))$ is
defined to mean: 
\begin{equation*}
\forall \epsilon > 0 \; \exists a_{\epsilon}, \; \forall i,j\geq
a_{\epsilon}, \; d(y,x_i,x_j)< \epsilon .
\end{equation*}
\end{definition}

Suppose $LIM(y,(x_i))$ and $LIM (y',(x_i))$. We would like to show that 
$d(y,y',x_i)\rightarrow 0$. However, this is not necessarily true: if $(x_i)$
is Cauchy then the properties $LIM$ are automatic (see Proposition \ref{dcontinuous} below).
So, we need to include the hypothesis that our sequence is not Cauchy, in the 
following statements. 

\begin{lemma}
\label{dlim}
Suppose $(x_i)$ is not Cauchy. If both $LIM(y,(x_i))$ and $LIM(y^{\prime},(x_i))$ hold,
then $d(y,y',x_i)\rightarrow 0$ as $i\rightarrow \infty$.
\end{lemma}

\begin{proof}
The sequence $(x_i)$ is supposed not to be Cauchy for $\varphi$, so there
exists $\epsilon _0>0$ such that for any $m\geq 0$ there are $i,j\geq m$
with $\varphi (x_i,x_j)\geq \epsilon _0$. Therefore, in view of the
definition of $\varphi$, for any $m$ there exist 
$i(m),j(m)\geq m$ and a point $z(m)\in X$ such that $d(x_{i(m)}, x_{j(m)},
z(m))\geq \epsilon _0/2$.

We now use condition (Trans) with $x=x_{i(m)}$ and $y=x_{j(m)}$ and $c=z(m)$,
for $a=y$ and $b=y^{\prime}$. This says 
\begin{equation*}
d(y,y^{\prime},x_{i(m)})\epsilon _0/2 \leq d(y,x_{i(m)}, x_{j(m)})+
d(y^{\prime},x_{i(m)}, x_{j(m)}).
\end{equation*}
If $LIM(y,(x_i))$, $LIM(y^{\prime},(x_i))$, then for any $\epsilon$ we can
assume $m$ is big enough so that 
\begin{equation*}
d(y,x_{i(m)}, x_{j(m)})\leq \epsilon \epsilon _0/4
\end{equation*}
and 
\begin{equation*}
d(y^{\prime},x_{i(m)}, x_{j(m)})\leq \epsilon \epsilon _0/4 .
\end{equation*}
Putting these together gives $d(y,y^{\prime},x_{i(m)})\leq \epsilon$. 

Choose $m$ so that for all 
$j,k\geq m$ we have $d(y,x_j,x_k)\leq \epsilon$ and the same for $y^{\prime}$. 
Then we have by (Tetr), for any $j\geq m$ 
\begin{equation*}
d(y^{\prime},y,x_j)\leq
d(x_{i(m)},y,x_j)+d(y^{\prime},x_{i(m)},x_j)+d(y^{\prime},y,x_{i(m)})\leq 3\epsilon .
\end{equation*}
Changing $\epsilon$ by a factor of three, we obtain the following statment:
for any $\epsilon > 0$ there exists $m$ such that for all $i\geq m$ we have 
$d(y^{\prime},y,x_i)\leq \epsilon$. This is the required convergence. 
\end{proof}

\begin{corollary}
\label{limcolin}
If the sequence $(x_i)$ is not
Cauchy for the distance $\varphi $, then the following property holds:
\newline
---if $LIM(y,(x_i))$, $LIM(y^{\prime},(x_i))$, and $LIM(y^{\prime\prime},(x_i))$
then $(y,y^{\prime},y^{\prime\prime})$ are colinear. 
\end{corollary}

\begin{proof}
We use the fact that 
\begin{equation*}
d(y,y^{\prime},y^{\prime\prime})\leq d(y,y^{\prime},x_i) +
d(y^{\prime},y^{\prime\prime},x_i)+d(y,y^{\prime\prime},x_i).
\end{equation*}
By Lemma \ref{dlim}, all three terms on the right approach $0$ as $i\rightarrow \infty$.
This proves that $d(y,y^{\prime},y^{\prime\prime})=0$.
\end{proof}

\begin{lemma}
\label{colinlim}
Suppose that $(x_i)$ is not Cauchy for the distance $\varphi $. If 
$LIM(y,(x_i))$, $LIM(y^{\prime},(x_i))$, $\varphi (y,y^{\prime})>0$ and 
$y^{\prime\prime}$ is a point such that $(y,y^{\prime},y^{\prime\prime})$ are
colinear, then also $LIM(y^{\prime\prime},(x_i))$.
\end{lemma}

\begin{proof}
By Lemma \ref{dlim}, 
for any $\epsilon > 0$ there exists $m$ such that for all $i\geq m$ we have 
$d(y^{\prime},y,x_i)\leq \epsilon$.

Let $u,v$ denote some $x_i$ or $x_j$. By hypothesis $\varphi (y,y^{\prime})>0$ so
there is a point $z$ such that 
$d(y,y^{\prime},z)=\epsilon _1>0$. Then condition (Trans) applied with 
$a=y^{\prime\prime}, b=u,x=y^{\prime},c=z, y=y$ gives 
\begin{equation*}
d(y^{\prime\prime},u,y^{\prime})d(z,y^{\prime},y)\leq
d(y^{\prime\prime},y^{\prime},y)+ d(u,y^{\prime},y).
\end{equation*}
Hence 
\begin{equation*}
d(y^{\prime\prime},u,y^{\prime})\leq d(u,y^{\prime},y)/\epsilon _1.
\end{equation*}
We can do the same for $v$, and also interchanging $y$ and $y^{\prime}$, to
get 
\begin{equation*}
d(y^{\prime\prime},v,y^{\prime})\leq d(v,y^{\prime},y)/\epsilon _1,
\end{equation*}
\begin{equation*}
d(y^{\prime\prime},u,y)\leq d(u,y,y^{\prime})/\epsilon _1,
\end{equation*}
\begin{equation*}
d(y^{\prime\prime},v,y)\leq d(v,y,y^{\prime})/\epsilon _1.
\end{equation*}
We have 
\begin{equation*}
d(y^{\prime\prime},u,v)\leq d(y,u,v) +
d(y^{\prime\prime},y,v)+d(y^{\prime\prime},u,y)
\end{equation*}
\begin{equation*}
\leq d(y,u,v) +(d(u,y,y^{\prime})+d(v,y,y^{\prime}))/\epsilon _1.
\end{equation*}
For $u=x_i$ and $v=x_j$ with $i,j\geq m$ as previously we get 
\begin{equation*}
d(y^{\prime\prime},x_i,x_j)\leq d(y,x_i,x_j)
+(d(x_i,y,y^{\prime})+d(x_j,y,y^{\prime}))/\epsilon _1.
\end{equation*}
By choosing $m$ big enough this can be made arbitrarily small, thus giving
the condition $LIM(y^{\prime\prime},(x_i))$.
\end{proof}

\begin{theorem}
\label{possibilities} 
Suppose that $(X,d)$ is a bounded $2$-metric space satisfying axiom (Trans) as above.
Suppose $(x_i)$ is a sequence. Then there are the
following possibilities (not necessarily exclusive): \newline
---there is no point $y$ with $LIM(y,(x_i))$; \newline
---there is exactly one point $y$ with $LIM(y,(x_i))$; \newline
---the sequence $(x_i)$ is Cauchy for the distance $\varphi$; or \newline
---the subset $Y\subset X$ of points $y$ such that $LIM(y,(x_i))$, is a line.
\end{theorem}

\begin{proof}
Consider the subset $Y\subset X$ of points $y$ such that $LIM(y,(x_i))$ holds.
We may assume that there are two distinct points $y_1\neq y_2$ in $Y$, 
for otherwise one of the first two possibilities
would hold. Suppose that $(x_i)$ is not
Cauchy for $\varphi$; in particular, Lemmas \ref{dlim}, \ref{colinlim} and Corollary
\ref{limcolin} apply.  

If $y,y^{\prime},y^{\prime\prime}$ are any three points in $Y$, then by 
Corollary \ref{limcolin}, they are colinear. Thus $Y$ is a subset satisfying 
Condition \eqref{subline} in the definition of a line; 
to show that it is a line, we have to show
that it is a maximal such subset. 

Suppose $Y\subset Y_1$ and $Y_1$ also satisfies
\eqref{subline}. 
Since $y_1\neq y_2$, and we are assuming that $\varphi$ satisfies
strict reflexivity (SR), we have $\varphi (y_1,y_2)\neq 0$. By Lemma \ref{positive},
$\varphi (y_1,y_2)>0$. If $y\in Y_1$ then by \eqref{subline},
$d(y,y_1,y_2)=0$. By Lemma \ref{colinlim}, 
$y$ must also satisfy $LIM(y,(x_i))$, thus $y\in Y$. This shows that
$Y_1\subset Y$, giving maximality of $Y$. Thus, $Y$ is a line. 
\end{proof}

The following proposition shows that the case when $(x_i)$ Cauchy has to be included in the
statement of the theorem. 

\begin{proposition}
\label{dcontinuous}
If $(x_i)$, $(y_j)$ and $(z_k)$ are Cauchy sequences, then the sequence $d(x_i,y_j,z_k)$
is Cauchy in the sense that for any $\epsilon >0$ there exists $M$ such that for $i,j,k,p,q,r\geq M$ then $|d(x_i,y_j,z_k) -d(x_p,y_q,z_r)|<\epsilon$. In particular
$d$ is continuous. If $(x_i)$ is Cauchy then $LIM(y,(x_i))$ holds for any point $y\in X$.
\end{proposition}

\begin{proof}
For given $\epsilon$, by the Cauchy condition there is $M$ such that 
for $i,j,k,p,q,r \geq M$ we have $\varphi (x_i,x_p)<\epsilon /6$, 
$\varphi (y_j,y_q)<\epsilon /6$,  and $\varphi (z_k,z_r)<\epsilon /6$.
Then by Lemma \ref{dphi}
\[
|d(x_i,y_j,z_k) -d(x_i,y_j,z_r)|\leq  \epsilon /3,  
\]
\[
|d(x_i,y_j,z_r) -d(x_i,y_q,z_r)|\leq  \epsilon /3 ,  
\]
and
\[
|d(x_i,q,z_r) -d(x_p,y_q,z_r)|\leq  \epsilon /3 .
\]
These give the Cauchy property
\[
|d(x_i,y_j,z_k) -d(x_p,y_q,z_r)|\leq  \epsilon .
\]
This shows in particular that $d$ is continuous. Suppose $(x_i)$ is Cauchy and
$y$ is any point. Then the sequence $d(y,x_i,x_j)$ is Cauchy in the above sense in the
two variables $i,j$, which gives exactly the condition $LIM (y,(x_i))$. 
\end{proof}

We say that a sequence $(x_i)$ is \emph{tri-Cauchy} if 
\begin{equation*}
\forall \epsilon > 0 ,\;\; \exists m_{\epsilon}, \;\; i,j,k \geq
m_{\epsilon} \Rightarrow d(x_i,x_j,x_k) < \epsilon .
\end{equation*}

\begin{lemma}
\label{accumulation}
Suppose $(x_i)$ is a tri-Cauchy sequence, and  $y\in X$ is an accumulation
point of the sequence with respect to the distance 
$\varphi$. Then $LIM(y,(x_i))$.
\end{lemma}

\begin{proof}
The condition that $y$ is an accumulation point means that there exists a
subsequence $(x_{u(k)})$ 
such that $(x_{u(k)})\rightarrow y$ with respect to the distance 
$\varphi$.
We have by (Tetr) 
\begin{eqnarray*}
d(y,x_i,x_j)& \leq & d(x_{u(k)},x_i,x_j) + d(y,x_{u(k)},x_j) + d(y,x_i,x_{u(k)})\\
& \leq & d(x_{u(k)},x_i,x_j) + 2 \varphi (y,x_{u(k)}),
\end{eqnarray*}
and both terms on the right become small, for $i,j$ big in the original
sequence and $k$ big in the subsequence. Hence $d(y,x_i,x_j)\rightarrow 0$
as $i,j\gg 0$, which is exactly the condition $LIM(y,(x_i))$.
\end{proof}

We say that $(X,d)$ is \emph{tri-complete} if, for any
tri-Cauchy sequence,  the set $Y$ of
points satisfying $LIM(y,(x_i))$ is nonempty. By Theorem \ref{possibilities},
$Y$ is either a single point, a line, or (in case $(x_i)$ is Cauchy) all of $X$. 

\begin{lemma}
\label{tricauchypossibilities}
Suppose  $(X,\varphi )$ is compact.  Then it is tri-complete, 
and for any tri-Cauchy sequence $(x_i)$ we have one of the following two possibilities:
\newline
---$(x_i)$ has a limit; or
\newline
---the subset $Y$ of points $y$ with $LIM(y,(x_i))$ is a line.
\end{lemma}

\begin{proof}
Suppose $(x_i)$ is a tri-Cauchy sequence. By compactness there is
at least one point of accumulation, so the set $Y$
of points $y$ with $LIM(y,(x_i))$ is nonempty 
by Lemma \ref{accumulation}. This rules out the first possibility of Theorem \ref{possibilities}. 

Suppose $Y$ consists of a single point $y$. We claim then that
$x_i\rightarrow y$. 
Suppose not: then there is a subsequence which doesn't contain $y$ in its
closure, but since $X$ is compact after going to a further subsequence we
may assume that the subsequence has a limit point $y^{\prime}\neq y$. But
again by Lemma \ref{accumulation}, 
we would have $LIM(y^{\prime},(x_i))$, a contradiction.
So in this case, the sequence $(x_i)$ is Cauchy for $\varphi$ and has $y$ as
its limit; thus we are also in the situation of the third possibility.
Note however that, since $(x_i)$ is Cauchy, the set of points $Y$ consists of
all of $X$ by Proposition \ref{dcontinuous}, so the second possibility doesn't occur unless
$X$ is a singleton. 

From Theorem \ref{possibilities} the remaining cases are that $(x_i)$ is Cauchy,
in which case it has a limit by compactness; or that $Y$ is a line.
\end{proof}

\section{Some examples}
\label{example}

The classic example of a $2$-metric is the function on ${\mathbb{R}} ^2$
defined by setting $d(x,y,z)$ to be the area of the triangle spanned by $x,y,z$.
Before discussing this example we look at a related example on 
${\mathbb{RP}} ^2$ which is also very canonical. 

\subsection{A projective area function}

Let $X:= S^2:= \{ (x_1,x_2,x_3)\in {\mathbb{R}} ^3, \;\; x_1^2 + x_2^2 +
x_3^2 = 1\}$. Define the function $d(x,y,z)$ by taking the absolute value of
the determinant of the matrix containing $x,y,z$ as column vectors: 
\begin{equation*}
d\left( \left[
\begin{array}{c}
x_1 \\ 
x_2 \\ 
x_3
\end{array}
\right] , \left[
\begin{array}{c}
y_1 \\ 
y_2 \\ 
y_3
\end{array}
\right] , \left[
\begin{array}{c}
z_1 \\ 
z_2 \\ 
z_3
\end{array}
\right] \right) := \left| \mathrm{det} \left[ 
\begin{array}{ccc}
x_1 & y_1 & z_1 \\ 
x_2 & y_2 & z_2 \\ 
x_3 & y_3 & z_3
\end{array}
\right] \right| .
\end{equation*}
This has appeared in Example 2.2 of \cite{MiczkoPalczewski}. 

\begin{proposition}
\label{determinantexample}
This function satisfies axioms (Sym), (Tetr), (Z), (B), (Trans).
\end{proposition}

\begin{proof}
Invariance under permutations (Sym) comes from the corresponding fact for
determinants. Condition (Z) comes from vanishing of a determinant with two
columns which are the same. Condition (B) comes from the fact that the
determinant of a matrix whose columns have norm $1$, is in $[-1, 1]$. We
have to verify (Tetr) and (Trans).

For condition (Tetr), suppose given vectors $x,y,z\in S^2$ as above, and
suppose that $d(x,y,z)>0$ i.e. they are linearly independent. Suppose given
another vector $a=\left[
\begin{array}{c}
a_1 \\ 
a_2 \\ 
a_3
\end{array}
\right] $ too. Then the determinants $d(a,y,z)$, $d(x,a,y)$ and $d(x,y,a)$
are the absolute values of the numerators appearing in Cramer's rule. This
means that if we write 
\begin{equation*}
a=\alpha x + \beta y + \gamma z
\end{equation*}
then 
\begin{equation*}
\frac{d(a,y,z)}{d(x,y,z)}= | \alpha |,
\end{equation*}
\begin{equation*}
\frac{d(x,a,z)}{d(x,y,z)}= | \beta |,
\end{equation*}
and 
\begin{equation*}
\frac{d(x,y,a)}{d(x,y,z)}= | \gamma |.
\end{equation*}
Now by the triangle inequality in ${\mathbb{R}} ^3$ we have 
\begin{equation*}
1=\| a \| \leq | \alpha | + | \beta | + | \gamma |
\end{equation*}
which gives exactly the relation (Tetr).

To prove (Trans), notice that it is invariant under orthogonal transformations
of ${\mathbb{R}} ^3$ so we may assume that 
\begin{equation*}
x= \left[
\begin{array}{c}
1 \\ 
0 \\ 
0
\end{array}
\right] , \;\; y = \left[
\begin{array}{c}
u \\ 
v \\ 
0
\end{array}
\right] , \;\; u^2 + v^2 = 1.
\end{equation*}
In this case, $\sup _{c\in S^2} d(c,x,y)= |v|$, so we are
reduced to considering 
\begin{equation*}
a= \left[
\begin{array}{c}
a_1 \\ 
a_2 \\ 
a_3
\end{array}
\right] , \;\; b = \left[
\begin{array}{c}
b_1 \\ 
b_2 \\ 
b_3
\end{array}
\right] .
\end{equation*}
Now $d(a,x,y)= |a_3v|$ and $d(b,x,y)=|b_3v|$, so we have to show that 
\begin{equation*}
d(a,b,x) |v| \leq |a_3 v| + |b_3v|.
\end{equation*}
But $d(a,b,x)= |a_2b_3 - a_3b_2|$ so this inequality is true (since 
$|a_2|\leq 1$ and $|b_2|\leq 1$). This completes the proof of (Trans).
\end{proof}

\begin{corollary}
If $X\subset S^2$ then with the same function $d(x,y,z)$, it still satisfies
the axioms.
\end{corollary}

\ \hfill $\Box$

The ``lines'' for the function $d$ defined above, are the great circles or
geodesics on $S^2$.

\begin{remark}
\label{antipode}
The distance function $\varphi$ is not strictly reflexive in this example,
indeed the associated equivalence relation identifies antipodal points. 
The quotient $S^2/\sim$ is the real projective plane. The corresponding function
on ${\mathbb{RP}} ^2$ is a bounded $2$-metric satisfying (Trans).
\end{remark}

\subsection{The Euclidean area function}

We next go back and consider the standard area function on Euclidean space, which we
denote by $\alpha $.

\begin{proposition}
\label{standardexample}
Let $U\subset {\mathbb{R}}^n$ be a ball of diameter $\leq 1$. 
Let $\alpha (x,y,z)$ be the area of the triangle spanned by $x,y,z\in U$.
Then $\alpha $ satisfies axioms (Sym), (Tetr), (Z), (N), (B), (Trans). 
\end{proposition}
\begin{proof}
Conditions (Sym), (Tetr), (Z), (N) are classical. Condition (B)
comes from the bound on the size of the ball. It remains to show Condition (Trans).
By invariance under orthogonal transformations, we may assume that $x=(0,\ldots ,0)$
and $y=(y_1,0,\ldots , 0)$ with $y_1>0$. 
Again by the bound on the size of the ball,
for any $c$ we have $\alpha (x,y,c)\leq y_1$. Thus, we need to show
\[
\alpha (x,a,b)y_1\leq \alpha (x,y,a)+\alpha (x,y,b).
\]
Write $a=(a_1,\ldots , a_n)$ and $b=(b_1,\ldots , b_n)$. Then
\[
\alpha (x,y,a)=y_1 \| \tilde{a}  \| 
\]
where $\tilde{a}= (0,a_2,\ldots , a_n)$ is the projection of $a$ on the plane 
orthogonal to $(xy)$. Similarly $d(x,y,b)=y_1 \| \tilde{b}  \| $.
To complete the proof it suffices to show that
$\alpha (x,a,b)\leq \| \tilde{a}  \|  + \| \tilde{b}  \| $.

Write $a=a_1e_1+\tilde{a}$ and $b=b_1e_1+\tilde{b}$ where
$e_1$ is the first basis vector. Then
\[
\alpha (x,a,b) = \frac{1}{2} \| a\wedge b \| 
\]
\[
= \frac{1}{2} \| a_1 e_1 \wedge \tilde{b}-b_1e_1\wedge \tilde{a} + 
\tilde{a}\wedge \tilde{b} \|
\]
\[
\leq \| \tilde{a}  \|  + \| \tilde{b}  \| 
\]
since $a_1\leq 1$, $b_1\leq 1$,   $\| \tilde{a}  \|\leq 1$, and $\| \tilde{b}  \|\leq 1$. 
\end{proof}

\subsection{Euclidean submanifolds}

If $U\subset {\mathbb{R}}^n$ is any subset contained in a ball of diameter $1$
but not in a single line,
then the function $d(x,y,z)$ on $U^3$  induced by the standard Euclidean area function on
${\mathbb{R}}^n$ is again a bounded isometric $2$-metric. This is interesting in case
$U$ is a patch inside a submanifold. The induced $2$-metric will have different
properties depending on the curvature of the submanifold. We consider 
$U\subset S^2\subset {\mathbb{R}}^3$ as an example. 

Choose a patch $U$ on $S^2$, contained in the neighborhood of the south pole of radius
$r<1/4$. Let $p:U\stackrel{\cong}{\rightarrow} V\subset {\mathbb{R}}^2$ be the vertical projection. 
Let $h(x,y,z):= \alpha (p^{-1}(x),p^{-1}(y),p^{-1}(z))$ be the $2$-metric on $V$ induced
from the standard Euclidean area function $\alpha $ of $S^2\subset {\mathbb{R}}^3$
via the projection $p$. This satisfies upper and lower bounds 
which mirror the convexity of the
piece of surface $U$:

\begin{lemma}
\label{s2bounds}
Keep the above notations for $h(x,y,z)$. 
Let $\alpha _2$ denote the standard area function on $V\subset {\mathbb{R}}^2$.
Define a function 
\[
\rho (x,y,z):= \| x-y\| \cdot \| x-z\| \cdot \| y-z\|  .
\]
Then there is a constant $C>1$ such that 
\begin{equation}
\label{convex}
\begin{array}{rcl}
\frac{1}{C}\left( \alpha _2 (x,y,z) + \rho (x,y,z) \right) 
& \leq & h(x,y,z) \\
h(x,y,z) & \leq &
C\left( \alpha _2 (x,y,z) + \rho (x,y,z) \right) 
\end{array}
\end{equation}
for $x,y,z\in V$. 
\end{lemma}
\begin{proof}
If any two points coincide the estimate holds, so we may assume the three points
$x,y,z$
are distinct, lying on a unique plane $P$. 
In particular, they lie on a circle $S^2\cap P$ of radius $\leq 1$. A simple calculation
on the circle gives a bound of the form $h(x,y,z)\geq (1/C)\rho (x,y,z)$.
The vertical projection from $P$ to ${\mathbb{R}}^2$ has Jacobian determinant $\leq 1$
so $h(x,y,z)\geq \alpha _2(x,y,z)$. Together these give the lower bound. 

For the upper bound, consider the unit normal vector to $P$ and let $u_3$ be the
vertical component. If $|u_3|\geq 1/4$ then the plane is somewhat horizontal and $h(x,y,z)\leq C\alpha _2(x,y,z)$. If $|u_3|\leq 1/4$ then the plane is almost vertical,
and in view of the fact that the patch is near the south pole, the intersection $P\cap S^2$ 
is a circle of radius $\geq 1/4$. In this case $h(x,y,z)\leq C\rho (x,y,z)$. From these
two cases we get the required upper bound. 
\end{proof}

\section{Fixed points of a map}
\label{fixedpoints}

Suppose that $(X,\varphi )$ is \emph{compact} meaning that every sequence
has a convergent subsequence. Suppose $F:X\rightarrow X$ is a $d$-decreasing
map i.e. one with $d(F(x),F(y),F(z))\leq kd(x,y,z)$ for $0<k<1$.

Pick a point $x_0\in X$ and define the {\em sequence of iterates} inductively by 
$x_{i+1}:= F(x_i)$. 

\begin{corollary}
Suppose $(X,\varphi )$ is compact and $F$ is a $d$-decreasing map. Pick a
point $x_0$ and define the sequence of iterates $(x_i)$ with $x_{i+1}=F(x_i)$.
This sequence is tri-Cauchy; hence either it is Cauchy
with a unique limit point $y\in X$, or else the subset 
$Y\subset X$ of points $y$ with $LIM(y,(x_i))$ is a line.
\end{corollary}

\begin{proof}
Note first that the sequence $(x_i)$ is tri-Cauchy. If $m\leq i,j,k$ then
\[
x_i=F^m(x_{i-m}), \;\; x_j=F^m(x_{j-m}), \;\;x_k=F^m(x_{k-m}).
\]
Hence using the global bound (B),
\[
d(x_i,x_j,x_k)\leq k^md(x_{i-m},x_{j-m},x_{k-m})\leq k^m.
\]
As $0<k<1$, for any $\epsilon$ there exists $m$ such that $k^m<\epsilon$; this gives
the tri-Cauchy property of the sequence of iterates. 

Then by Lemma \ref{tricauchypossibilities}, 
either $(x_i)\rightarrow y$ or else the set $Y$ of points $y$ with
$LIM(y,(x_i))$ is a line. 
\end{proof}

\begin{theorem}
\label{fixptline}
Suppose $X$ is nonempty and
$d$ is a bounded transitive $2$-metric (i.e. one satisfying (B) and (Trans)), 
such that $(X,\varphi )$ is compact. Suppose $F$ is a $d$-decreasing map
for a constant $0<k<1$. Then, either $F$ has a fixed point, or there is a line $Y$ fixed in
the sense that $F(Y)\subset Y$ and $Y$ is the only line containing $F(Y)$.
\end{theorem}

\begin{proof}
Pick $x_0\in X$ and define the sequence of iterates $(x_i)$ as above. By the previous
corollary, either $x_i\rightarrow y$ or else the set $Y$ of points $y$ with
$LIM (y,(x_i))$ is a line. 

Suppose we are in the second possibility but not the first; 
thus $(x_i)$ has more than one accumulation point. Suppose $y$ is
a point with property $LIM(y,(x_i))$, which means that $d(y,x_i,x_j)<\epsilon$ for $i,j\geq m_{\epsilon}$.
Hence if $i,j\geq m_{\epsilon}+1$ then 
\[
d(F(y),x_i,x_j)=d(F(y),F(x_{i-1}),F(x_{j-1}))<k\epsilon .
\]
This shows the property $LIM(F(y),(x_i))$. Thus, $F(Y)\subset Y$. 

Suppose $Y_1$ is a line with $F(Y)\subset Y_1$. If $F(Y)$ contains at least
two distinct points then there is at most one line containing $F(Y)$
by Lemma \ref{lineunique} and we obtain the second conclusion of
the theorem. Suppose $F(Y)=\{ y_0\}$ consists of a single point. 
Then, $y_0\in Y$ so $F(y_0)\in F(Y)$, which shows that $F(y_0)=y_0$; in this
case $F$ admits a fixed point. 

We may now assume that we are in the first case of the first paragraph: $x_i\rightarrow y$.
If $F(y)=y$ we have a fixed point, so assume\footnote{We need to consider this
case: as 
$F$ is not assumed to be surjective, it is not necessarily continuous for $\varphi$
so the convergence of the sequence of iterates towards $y$ doesn't directly
imply that $y$ is a fixed point.}
 $F(y)\neq y$. Let $Y$ be the
unique line containing $y$ and $F(y)$ by Lemma \ref{lineunique} which also says $Y$ is the
set of points colinear with $y$ and $F(y)$. 

We claim that $F(X)\subset Y$.
If $z\in X$ then $LIM(z,(x_i))$ by the last part of Proposition \ref{dcontinuous}.
For a given $\epsilon$ there is $m_{\epsilon}$ such that $i,j\geq m_{\epsilon}\Rightarrow
d(x_i,x_j,z)<\epsilon$. However, for $i$ fixed we have $x_j\rightarrow y$ by hypothesis,
and the continuity of $d$ (Proposition \ref{dcontinuous}) implies that $d(x_i,y,z)<\epsilon$. Apply $F$, giving 
\[
\forall i\geq m_{\epsilon},\;\;\; d(x_{i+1},F(y),F(z))<k\epsilon .
\]
Again using continuity of $d$, we have 
\[
d(x_{i+1},F(y),F(z))\rightarrow d(y,F(y),F(z))
\]
and the above then implies that $d(y,F(y),F(z))<k\epsilon $ for any $\epsilon >0$. 
Hence $d(y,F(y),F(z))=0$ which means $F(z)\in Y$. This proves that $F(X)\subset Y$
{\em a fortiori} $F(Y)\subset Y$. If $F(F(y))$ is distinct from $F(y)$ then
$Y$ is the unique line containing $F(Y)$, otherwise $F(y)$ is a fixed point of $F$.
This completes the proof of the theorem. 
\end{proof}

\begin{corollary}
Suppose $X\subset S^2$ is a closed subset. Define the function $d(x,y,z)$ by a determinant
as in Proposition \ref{determinantexample}. If $F:X\rightarrow X$ is a function such that $d(F(x),F(y),F(z))\leq
kd(x,y,z)$ for $0<k<1$ then either it has a fixed pair of antipodal points, or a fixed great
circle.
\end{corollary}

\begin{proof}
Recall that we should really be working with the
image of $X$ in the real projective space $\rr \pp ^2=S^2/\sim$ (Remark \ref{antipode}).
On this quotient, the previous theorem applies. Note that by a ``fixed great circle''
we mean a subset $Y\subset X$ which is the intersection of $X$ with a great circle,
and such that $Y$ is the intersection of $X$ with the unique great circle 
containing $F(Y)$. 
\end{proof}

\section{Examples of contractive mappings}
\label{examplemaps}

Here are some basic examples which show that the phenomenon pointed out
by Hsiao \cite{Hsiao} doesn't affect our notion of contractive mapping.
He showed that for a number of different contraction conditions for a mapping
$F$ on a $2$-metric space, the iterates $x_i$ defined by $x_{i+1}:= F(x_i)$ starting
with any $x_0$, are all colinear i.e. $d(x_i,x_j,x_k)=0$ for all $i,j,k$. 

\subsection{}
Consider first the standard $2$-metric given by the area of the spanned triangle
in ${\mathbb{R}}^n$, see Proposition \ref{standardexample}. Suppose $U$
is a convex region containing $0$, contained in a ball of diameter $\leq 1$. Choose $0<k<1$
and put
$G(x_1,\ldots , x_n):= (kx_1,\ldots , kx_n)$. Then clearly 
\[
d(Gx,Gy,Gz)\leq k^2d(x,y,z).
\]
Even though $G$ itself satisfies the property of colinearity of iterates,
suppose $M\in O(n)$ is a linear orthogonal transformation of ${\mathbb{R}}^n$
preserving $U$. Then $d(Mx,My,Mz)=d(x,y,z)$, so if we  put $F(x):= M(G(x))$ then
\[
d(Fx, Fy,Fz)\leq k^2d(x,y,z).
\]
Hence $F$ is a contractive mapping. If $M$ is some kind of rotation for example,
the iterates $F^i(x_0)$ will not all be colinear, so Hsiao's conclusion doesn't hold
in this case. Of course, in this example there is a unique fixed point $0\in U$ so
the phenomenon of fixed lines hasn't shown up yet. 

\subsection{}
Look now at the determinant $2$-metric on ${\mathbb{RP}}^2$ of Proposition 
\ref{determinantexample}, viewed for convenience
as a function on $S^2$. 
Fix $0<e<1$ and consider a subset $U\subset S^2$ defined by the condition 
\[
U= \{ (x_1,x_2,x_3)\subset S^2,\;\;\mbox{s.t.}\; \| (x_1,x_2,0\| \geq e \} .
\]
This is a tubular neighborhood of the equator $x_3=0$. Choose $0<k<1$ and put
\[
G(x_1,x_2,x_3):= \frac{ (x_1,x_2,kx_3)}{\| (x_1,x_2,kx_3)\| }= 
\frac{ \tilde{x}}{\| \tilde{x} \| },
\]
where $\tilde{x}:= (x_1,x_2,kx_3)$.
The mapping $G$ preserves $U$, in fact it is a contraction towards the equator. 
Note that $\| \tilde{x}\| \geq \| (x_1,x_2,0)\|\geq e$,
so 
\[
d(Gx,Gy,Gz) = \frac{kd(x,y,z)}{\|\tilde{x}\| \, \|\tilde{y}\|\, \|\tilde{z}\|}
\leq ke^{-3}d(x,y,z) .
\]
If $k<e^3$ then this mapping is contractive. Every point of the equator is
a fixed point, and indeed $G$ is interesting from the point of view of studying
uniqueness, since it has the equator (containing all of the fixed points), as well as all of the longitudes (each containing a single fixed point), as fixed lines.

To get an example where there are no fixed points but a fixed line,
let $M\in O(3)$ be any orthogonal linear transformation of 
${\mathbb{R}}^3$, which preserves $S^2$. Suppose that $M$ preserves $U$; this will be the
case for example if $M$ is a rotation preserving the equator. Then $M$ preserves the
$2$-metric $d$, so $F(x):= M(G(x))$ is a contractive mapping $U\rightarrow U$. 
If $M$ is a nontrivial rotation, then it $F$ has no fixed points, but the equator is a fixed line. It is easy to see that the iterates $F^i(x_0)$ are not colinear in general. 
Thus Hsiao's remark \cite{Hsiao} doesn't apply to our notion of contractive mapping. 

\subsection{}
Come back to the standard area $2$-metric on ${\mathbb{R}}^n$. 
It should be pointed out that our contractivity condition automatically
implies that $F$ preserves the relation of colinearity, hence it preserves
lines. So, for example, if $U$ is a compact region in ${\mathbb{R}}^n$ of diameter
$\leq 1$, and if $U$ has dense and connected interior, then any contractive
mapping $F:U\rightarrow U$ in our sense must be the restriction of a projective
transformation of ${\mathbb{RP}}^n$, that is given by an element of $PGL(n+1,{\mathbb{R}})$. 
In this sense the Euclidean $2$-metric
itself has a strong rigidity property. 

On the other hand, we can easily consider examples
which don't contain any lines with more than two points. 
This is the case if $U$ is contained in the boundary of a strictly convex region. 
Furthermore, that reduces the possibility of
a fixed line in Theorem \ref{fixptline} to the case of two points
interchanged, whence:

\begin{corollary}
Suppose $U$ is a compact subset of the boundary of a strictly convex region of
diameter $\leq 1$ in ${\mathbb{R}}^n$, and let $d$ be the restriction of the
standard area $2$-metric (Proposition \ref{standardexample}). If
$F:U\rightarrow U$ is any mapping such that $d(F(x),F(y),F(z))\leq kd(x,y,z)$ for
a constant $0<k<1$, then either $F$ has  a fixed point, or else it interchanges
two points in a fixed pair.
\end{corollary}
\begin{proof}
In the boundary of a strictly convex region, the ``lines'' are automatically subsets
consisting of only two points. 
\end{proof}

\subsection{}
For $(U,d)$ as in the previous corollary, for example a region on the sphere $S^2$ with the
restriction of the area $2$-metric from ${\mathbb{R}}^3$ as in Lemma \ref{s2bounds}, 
there exist many contractive mappings $F$. It requires a little bit of calculation to
show this. 

Suppose $V'\subset V \subset {\mathbb{R}}^2$ are open sets. For a $C^2$ mapping
$F:V'\rightarrow V$, let $J(F,x)$ denote the Jacobian matrix of $F$ at a point $x$,
viewed as an actual $2\times 2$ matrix 
using the standard frame for the tangent bundle of ${\mathbb{R}}^2$.
Let $dJ(F)$ denote the derivative of this matrix, i.e. the Hessian matrix of $F$.

\begin{proposition}
\label{s2maps}
Suppose $d$ is a $2$-metric on $V$ satisfying a convexity
bound of the form \eqref{convex}
$$
\frac{1}{C}\left( \alpha _2 (x,y,z) + \rho (x,y,z) \right) 
\leq d(x,y,z) \leq 
C\left( \alpha _2 (x,y,z) + \rho (x,y,z) \right) 
$$
for a constant $C>1$.
Such $d$ exists by Lemma \ref{s2bounds}. Given $C_A>0$, there is
a constant $C'>0$
such that the following holds. 

Suppose $A$ is a fixed $2\times 2$ matrix
with ${\rm det}(A)\neq 0$ and $\| A\| \leq C_A$.
Then there is a constant $c'>0$ depending on $A$, 
such that if $F:V'\rightarrow V$ is a $C^2$ mapping
with $\| J(F,x)-A\| \leq c'$ and $\| dJ(F)\| \leq c'$ over $V'$, then 
\[
d(F(x),F(y),F(z))\leq C'|{\rm det}(A)|d(x,y,z)
\]
for any $x,y,z\in V'$. 
\end{proposition}
\begin{proof}
In what follows the constants $C_i$ and $C'$ will depend only on $C$ and 
the bound $C_A$ for $\| A\| $,
assuming $c'$ to be chosen small enough depending on $A$. 
By choosing $c'$ small enough we may assume that $F$ is locally a 
diffeomorphism, as $J(F,x)$ is invertible, being close to the invertible matrix $A$. 
By Rolle's theorem 
there is a point $y'$ on the segment joining $x$ to $y$, and a positive real
number $u$, such that 
\[
F(y)-F(x) = u J(F,y')(y-x).
\]
Furthermore $u<C_1$, if $c'$ is small enough. We thank N. Mestrano for 
this argument.

Similarly there is a point $z'$ on the segment joining $x$ to $z$, 
and a positive real number $v<C_1$ such that 
\[
F(z)-F(x) = v J(F,z')(z-x).
\]
Put $S:= J (F,y')-J(F,x)$ and $T:= J (F,z')-J(F,x)$.
Now the bound $\| dJ(F)\| \leq c'$ implies that 
\[
\| S\|  \leq 4c' \| y-x\|  ,
\;\;\; 
\|  T\|  \leq 4c' \| z-x\|  .
\]
We have 
\[
\alpha  _2(F(x),F(y),F(z)) = \left| (F(y)-F(x))\wedge (F(z)-F(x))\right| 
\]
\[
= \left|(uJ(F,y')(y-x))\wedge (vJ(F,z')(z-x))\right|
\]
\[
= uv\left| (J(F,x)+S)(y-x)\wedge (J(F,x)+T)(z-x)\right| 
\]
\[
\leq uv  |{\rm det}J(F,x)| \alpha _2(x,y,z) 
\]
\[
\mbox{   }
+ uv[(\| S\| +\| T\|  ) \| J(F,x)\|  + \| S\| \| T\| ]\cdot 
\| y-x\| \cdot \| z-x\|  . 
\]
Note that 
\[
\| y-x\| ^2\| z-x\|  +\| y-x\| \| z-x\| ^2 + \| y-x\| ^2\| z-x\| ^2\leq C_2\rho (x,y,z),
\] 
and $\| J(F,x)\|  \leq C_3$. 
By chosing $c'$ small enough depending on $A$, we may assume that 
$|{\rm det}J(F,x)|\leq C_4 |{\rm det}A|$. So again possibly by reducing $c'$, 
we get
\[
\alpha  _2(F(x),F(y),F(z)) \leq C_4|{\rm det}(A)| \alpha _2(x,y,z) +
C_1C_2C_3\rho (x,y,z).
\]
It is also clear that $\rho (F(x),F(y),F(z))\leq C_5\rho (x,y,z)$.
Using the convexity estimate in the hypothesis of the proposition,
we get
\[
d(F(x),F(y),F(z))\leq C'|{\rm det}(A)|d(x,y,z)
\]
for a constant $C'$ which depends on $C$ and $C_A$ but not on $A$ itself.
\end{proof}

Suppose $V$ is a disc centered at the origin in ${\mathbb{R}}^2$ with a $2$-metric $d$
satisfying the convexity estimate of the form \eqref{convex}. For example
$V$ could be the projection
of a patch in the Euclidean $S^2$ as in Lemma \ref{s2bounds}.
Fix a bound $C_A=2$ for example. We get a constant $C'$ from the previous 
proposition. Choose then a matrix $A$ such that 
$C'|{\rm det}(A)| < 1$ and $A\cdot V\subset V$. Then there is $c'$
given by the previous proposition. There exist plenty of $C^2$ mappings 
$F:V\rightarrow V$ with $\| J(F,x)-A\| \leq c'$ and $\| dJ(F)\| \leq c'$.
By the conclusion of the proposition, these are contractive. 

This example shows in a strong sense that there are no colinearity constraints such as
\cite{Hsiao} for mappings which are contractive in our sense.

As a $2$-metric moves from a convex one towards the flat Euclidean one
(for example for patches of the same size but on spheres with bigger and bigger radii)
we have some kind of a deformation from a nonrigid situation to a rigid one.
This provides a model for investigating this general 
phenomenon in geometry.

The question of understanding the geometry of contractive mappings, or
even mappings $F$ with $d(F(x),F(y),F(z))\leq Kd(x,y,z)$ for any constant $K$,
seems to be an interesting geometric problem. 
Following the extensive literature in this subject,
more complicated situations involving several compatible mappings and additional
terms in the contractivity condition may also be envisioned.


\end{document}